\documentclass[a4paper, 12pt]{amsproc}
\usepackage{amssymb,amsthm,amsfonts,latexsym}
\newtheorem{theorem}{Theorem}[section]
\newtheorem{lemma}[theorem]{Lemma}

\theoremstyle{remark}

\title[Non-Trivial Kazhdan-Lusztig Coefficients]
{Some Non-Trivial Kazhdan-Lusztig Coefficients
 of an   Affine
Weyl Group \\ of Type $\tilde A_n$}

\author[L. Scott and N. Xi]{Leonard Scott$^{*}$ and Nanhua Xi$^{\dagger}$}
\address{$^{*}$
Department of Mathematics\\
University of Virginia \\
Charlottesville, VA22903 \\
USA} \email{lls2l@virginia.edu}
\address{$^{\dagger}$
Hua Loo-Keng Key Laboratory of Mathematics and
Institute of Mathematics\\
Chinese Academy of Sciences\\
Beijing, 100190\\
China } \email{nanhua@math.ac.cn}
\thanks{$*$ L. Scott was supported
by NSF}
\thanks{$\dagger$ N. Xi was partially supported by Natural Sciences Foundation
of China (No. 10671193).}

\begin{document}
\baselineskip=18pt
\begin{abstract}
In this paper we show that the leading coefficient $\mu(y,w)$ of
some Kazhdan-Lusztig polynomials $P_{y,w}$ with $y,w$ in an affine
Weyl group of type $\tilde A_n $ is $n+2$. This fact has some
consequences on the dimension of first extension groups of finite
groups of Lie type with irreducible coefficients.

\end{abstract}

\maketitle

\def\Cal{\mathcal}
\def\bold{\mathbf}
\def\ca{\mathcal A}
\def\cdz{\mathcal D_0}
\def\cd{\mathcal D}
\def\cdo{\mathcal D_1}
\def\bold{\mathbf}
\def\l{\lambda}
\def\le{\leq}

Given two elements $y\leq w$ in a Coxeter group $(W,S)$ ($S$ the set
of simple reflections), we have a Kazhdan-Lusztig polynomial
$P_{y,w}$ in an indeterminate $q$. If $y<w$, the degree of $P_{y,w}$
is less than or equal to $\frac12(l(w)-l(y)-1)$. Particularly
interesting is the coefficient $\mu(y,w)$ of the term
$q^{\frac12(l(w)-l(y)-1)}$ in $P_{y,w}$, since it plays a key role
in understanding Kazhdan-Lusztig polynomials and in a recursive
formula for them. Moreover, this ``leading"  coefficient (it can be
zero) is important in representation theory and in understanding
cohomology and first extension groups for irreducible modules of
algebraic groups and of finite groups of Lie type.

However, it is in general hard to compute the leading coefficient.
In \cite{L6} Lusztig computes the leading coefficient for some
Kazhdan-Lusztig polynomials of an affine Weyl group of type $\tilde
B_2$, more are computed in \cite{W}. In \cite{S} for an affine Weyl
group of type $\tilde A_5$, some non-trivial leading coefficients
are worked out. McLarnan and Warrington have shown that $\mu(y,w)$
can be greater than 1 for symmetric groups, see \cite{MW}. In
\cite{X3}, Xi shows that if $a(y)<a(w)$, then $\mu(y,w)\le 1$ when
$W$ is a symmetric group or an affine Weyl group of type $\tilde
A_n$.

In this paper we  show that the leading coefficient $\mu(y,w)$ of
some Kazhdan-Lusztig polynomials $P_{y,w}$ with $y,w$ in an  affine
Weyl group of type $\tilde A_n $ is $n+2$ (see Theorem 3.3). There
is a well-known identification \cite{A1} of this coefficient with
dimensions of first extension groups for irreducible modules of the
underlying algebraic group, which here is $SL_{n+1}(\bar{\mathbb
F}_p)$, in the presence of the Lusztig conjecture (known to hold for
$p$ very large \cite{AJS}). Thus, our results show the dimensions of
these first extension groups can be arbitrarily large as $n$ becomes
large. Taken together with \cite{CPS2}, this implies that the
corresponding first extension groups for the finite groups
$SL_{n+1}(\mathbb F_q)$, $q$ a sufficiently large power of (a
sufficiently large) prime $p$, must also have unbounded dimensions.
In particular, a well-known conjecture of Robert Guralnick \cite{G},
that there exists a universal constant bound on dimensions of the
first cohomology groups of finite groups (with faithful absolutely
irreducible modules as coefficients), cannot be extended to first
extension groups.

In Section 5 we  give a representation-theoretic approach to (a
variation on) Theorem 3.3. It does not yield the same precise
calculation\footnote{H. Andersen has recently provided a way to
obtain precise formulas, as in Theorem 3.3, from the
representation-theoretic approach of Section 5 (but still using the
Coxeter group lemma, Lemma 3.4), by using some homological results
of \cite{AJ}. We sketch Andersen's argument in Remark 5.3 (d).}, but
applies to more weights; see Remark 5.3(c). More importantly, it
yields an independent confirmation of the fact demonstrated by
Theorem 3.3, that the coefficients $\mu(y,w)$, and the dimensions of
first extension groups which correspond to them, can go to infinity
with $n$.

\section{Springer's formula}

In this section we recall some basic facts and a formula of Springer
for the leading coefficient $\mu(y,w)$.

\subsection{}
  Let $G$ be a
connected, simply connected reductive algebraic group over the field
{\bf C} of complex numbers  and $T$ a maximal torus of $G$. Let
$N_G(T)$ be the normalizer of $T$ in $G$. Then $W_0=N_G(T)/T$ is a
Weyl group, which acts on the character group $X={\rm{Hom}}(T,\bold
C^*)$ of $T$. The semi-direct product $W_0\ltimes X$ is  called an
extended affine Weyl group, denoted by $W$. It contains the affine
Weyl group $W_a$, the semi-direct product of $W_0$ and the root
lattice. We shall denote by $S$ the set of simple reflections of
$W$. We shall denote the length function of $W$ by $l$ and use
$\leq$ for the Bruhat order on $W$. We refer to subsection 2.1 for a
formula of the length function, see also subsections 1.1 and 1.2 in
\cite{L5} or section 1.1 in \cite{X2} for the length function and
Bruhat order. See also \cite{L2}, where Lusztig explains how to
carry over notions from \cite{KL}, including Kazhdan-Lusztig
polynomials, to $(W,S)$ and a Hecke algebra for it.

Let $H$ be the  Hecke algebra  of $(W,S)$ over $\Cal A=\bold
Z[q^{\frac 12},q^{-\frac 12}]$ $(q$ an indeterminate) with parameter
$q$. Let  $\{T_w\}_{w\in W}$ be its standard basis. Let
$C_w=q^{-\frac {l(w)}2}\sum_{y\le w}P_{y,w}T_y,\ w\in W$ be the
Kazhdan-Lusztig basis of $H$, where $P_{y,w}$ are the
Kazhdan-Lusztig polynomials. The degree of $P_{y,w}$ is less than or
equal to $\frac12(l(w)-l(y)-1)$ if $y<w$. We write
$P_{y,w}=\mu(y,w)q^{\frac12(l(w)-l(y)-1)}$+lower degree terms. The
coefficient $\mu(y,w)$ is very interesting, this can be seen even
from the recursive formula (see \cite{KL}) for Kazhdan-Lusztig
polynomials. We shall call $\mu(y,w)$  the {\it Kazhdan-Lusztig
coefficient} of $P_{y,w}$. The extended and usual (non-extended)
affine Weyl groups have essentially the same Kazhdan-Lusztig
polynomials. For more details about Hecke algebras of extended
affine Weyl groups, we refer to Section 4 in \cite{L2}, or
Subsection 1.2 in \cite{L5}, or Sections 1.1 and 1.6 in \cite{X2}.

\subsection{} Write $$C_xC_y=\sum_{z\in W}h_{x,y,z}C_z,\qquad
h_{x,y,z}\in \mathcal A
 =\bold Z[q^{\frac 12},q^{-\frac 12}].$$
 Following Lusztig (\cite{L3}), we define
 $$a(z)={\rm min}\{i\in\bold N\ |\ q^{-\frac i2}h_{x,y,z}\in\bold Z[q^{-\frac
 12}]{\rm\ for \ all\ }x,y\in W\}.$$ If for any $i$,
 $q^{-\frac i2}h_{x,y,z}\not\in\bold Z[q^{-\frac
 12}]{\rm\ for \ some\ }x,y\in W$, we set $a(z)=\infty.$
 Then $a(w)\le
l(w_0)$ for any $w\in W$, where $w_0$ is the longest element of
$W_0$ (see \cite{L3}]).

Following Lusztig and Springer,
 we define $\delta_{x,y,z}$ and $\gamma_{x,y,z}$ by the following formula,
 $$h_{x,y,z}=\gamma_{x,y,z}q^{\frac {a(z)}2}+\delta_{x,y,z}q^{\frac
 {a(z)-1}2}+
 {\rm\ lower\ degree\ terms}.$$
Springer showed that $l(z)\ge a(z)$ (see \cite{L4}). Let $\delta(z)$
be the
 degree of $P_{e,z}$, where $e$ is the neutral element of $W$.
 Then actually one has $l(z)-a(z)-2\delta(z)\ge 0$ (see \cite{L4}). Set
 $$\cd_i=\{z\in W\ |\ l(z)-a(z)-2\delta(z)=i\}.$$The number $\pi(z)$ is
 defined by $P_{e,z}=\pi(z)q^{\delta(z)}+
 {\rm\ lower\ degree\ terms}.$

 The elements of $\cdz$ are involutions, called distinguished involutions of
 $(W,S)$ (see \cite{L4}).

\def\ll{\underset {L}{\leq}}
\def\rl{\underset {R}{\leq}}
\def\lrl{\underset {LR}{\leq}}
\def\llr{\lrl}
\def\el{\underset {L}{\sim}}
\def\er{\underset {R}{\sim}}
\def\elr{\underset {LR}{\sim}}
\def\ds{\displaystyle\sum}

 \subsection{}  Assume that $(W,S)$ is an extended affine Weyl group or a Weyl
group. The
 following formula is due to Springer \cite{Sp}(see \cite{X2} for a sketchy proof),

$$\begin{array}{rl}\mu(y,x)&=\displaystyle\sum_{d\in\cdz}\delta_{y^{-1},x,d}+
 \displaystyle\sum_{f\in\cdo}\gamma_{y^{-1},x,f}\pi(f)\\[3mm] &=\displaystyle\sum_{d\in\cdz}\delta_{y,x^{-1},d}+
 \displaystyle\sum_{f\in\cdo}\gamma_{y,x^{-1},f}\pi(f).\end{array}$$

 \medskip
\def\vp{\varphi}
\def\st{\stackrel}
\def\sc{\scriptstyle}

\subsection{}  We refer to \cite{KL} for the definition of the
preorders $\ll,\ \rl,\lrl$ and of the equivalence relations $\el,\
\er,\ \elr$ on $W$. The corresponding equivalence classes are called
{ left cells, right cells, two-sided cells} of $W$, respectively.
The preorder $\ll$ (resp. $\rl;\lrl$) induces a partial order on the
set of left (resp. right; two-sided) cells of $W$, denoted again by
$\ll$ (resp. $\rl;\lrl$). For a Weyl group or an extended affine
Weyl
 group, Springer showed [Sp]
the following results (a) and (b) (see \cite{X3})

\medskip

\noindent (a) Assume that $\mu(y,w)$ or $\mu(w,y)$ is nonzero,
then
 $y\ll w$ and $y\rl w$ if $a(y)<a(w)$, and $y\el w$ or $y\er w$ if
 $a(y)=a(w)$.

\medskip

\noindent (b) If $\delta_{x,y,z}\ne 0$, then $z\el y$ or $z\er x$.
 (Note that  $h_{x,y,z}\ne 0$ implies that $a(z)\ge a(x)$ and
 $a(z)\ge a(y)$, see \cite{L3}.)

\medskip

For $w\in W$, set $L(w)=\{s\in S\ |\ sw\leq w\}$, \ $R(w)=\{s\in S\
|\ ws\leq w\}.$ Then we have (see \cite{KL})

\medskip

\noindent(c) $R(w)\subseteq R(y)$ if $y\ll w$. In particular,
$R(w)=R(y)$ if $y\el w$;

\medskip

\noindent(d) $L(w)\subseteq L(y)$ if $y\rl w$. In particular,
$L(w)= L(y)$ if $y\er w$.

\section{The lowest two-sided cell}

In this section we collect some facts about the lowest two-sided
cell of $W$.

\def\gz{\Gamma_0}

\subsection{} Let $w_0$ be the longest element of $W_0$. Let
$$\Gamma_0=\{ww_0\,|\, w\in W,\ \ l(ww_0)=l(w)+l(w_0)\}.$$ Then
$\gz$ is a left cell (see \cite{L3}). The Kazhdan-Lusztig
polynomials $P_{y,w}$ for $y,w$ in $\Gamma_0$ play a key role in
Lusztig's conjectures on irreducible characters of algebraic groups,
of quantum groups at roots of unity and of affine Lie algebras. In
this paper we are interested in the Kazhdan-Lusztig coefficient of
$P_{y,w}$ for $y,w$ in $\gz$.

It is known (see \cite{Sh1}) that $c_0=\{w\in W|a(w)=l(w_0)\}$ is a
two-sided cell and contains $\gz$. In fact, $c_0$ is the lowest
two-sided cell with respect to the partial order $\lrl$ on the set
of two-sided cells of $W$ and  $c_0$ has $|W_0|$ left cells (see
\cite{Sh2}).

\def\st{\stackrel}
\def\sc{\scriptstyle}
Let  $R^+$ (resp. $R^-$, $\Delta$) be the set of positive (resp.
negative, simple) roots of  the root system $R$ of $W_0$. Then the
length of $xw$ ($w\in W_0,\ x\in X)$ is given by the formula (see
\cite{IM}) $$l(xw)=\displaystyle\sum_{\st{\sc \alpha\in R^+}
{w(\alpha)\in R^-}}|\langle
x,\alpha^\vee\rangle+1|+\displaystyle\sum_{\st{\sc \alpha\in R^+}
{w(\alpha)\in R^+}}|\langle x,\alpha^\vee\rangle|.$$Let $X^+=\{x\in
X|l(xw_0)=l(x)+l(w_0)\}$ be the set of dominant weights of $X$. For
each simple root $\alpha$ we denote by $s_\alpha$  the corresponding
simple reflection in $ W_0$ and $x_\alpha$ the corresponding
fundamental weight. Then we have $s_\alpha(y)=y-\langle
y,\alpha^\vee\rangle\alpha$ for any $y\in X$ and $\langle
x_\alpha,\beta^\vee\rangle=\delta_{\alpha\beta}$ for any simple
roots $\alpha,\ \beta$.  For each $w\in W_0$, we set
$$d_w=w\prod_{\st {\alpha\in\Delta}{ w(\alpha)\in R^-}}x_\alpha.$$ Then
$$c_0=\{d_wxw_0d_u^{-1}|w,u\in W_0, x\in X^+\}.$$
Moreover, the set $c'_{0,w}=\{d_wxw_0d_u^{-1}|u\in W_0, x\in X^+\}$
is a right cell of $W$ and $c_{0,u}=\{d_wxw_0d_u^{-1}|w\in W_0, x\in
X^+\}$ is a left cell of $W$. The distinguished involutions of $c_0$
are $d_ww_0d_w^{-1},\ w\in W_0$. The set $\{d_w\,|\, w\in W_0\}$ can
also be described as $\{z\in W\,|\,zw_0\in c_0 \text{ but }
zw_0s\not\in c_0 \text{ for any } s\in W_0-\{e\}\}.$ See \cite{Sh2}.

\subsection{}   For $x\in X^+$ let
$V(x)$ be a rational irreducible $G$-module of highest weight $x$
and let $S_x$ be the corresponding element defined in \cite{L2}.
Then $S_x,\ x\in X^+$, form an $\Cal A$-basis of the center of $H$.
For $w\in W_0$ we define
$$E_{d_w}=q^{-\frac {l(d_w)}2}\sum_{\st{\sc y\le d_w}
{l(yw_0)=l(y)+l(w_0)}} P_{yw_0,d_ww_0}T_y$$ and
$$F_{d_w}=q^{-\frac {l(d_w)}2}\sum_{\st{\sc y\le d_w}
{l(yw_0)=l(y)+l(w_0)}} P_{yw_0,d_ww_0}T_{y^{-1}}.$$ Then we have
(see \cite[Corollary 2.11]{X1} and \cite[Proposition 8.6]{L2})

(a) $E_{d_w}S_xC_{w_0}F_{d_u}=C_{d_wxw_0d_u^{-1}}$ for any $w,u\in
W$ and $x\in X^+$.

(b) $S_xS_y=\sum_{z\in X^+}m_{x,y,z}S_z$ for any $x,y\in X^+$.
Here $m_{x,y,z}$ is defined to be the multiplicity of $V(z)$ in
the tensor product $V(x)\otimes V(y)$.
\def\Hom{{\text{Hom}}}

Using (a), (b) and 1.3 we get

(c) Let $w,w',u\in W_0,\ y,z\in X^+$ and let $d_wyw_0d_u^{-1}, \
d_{w'}zw_0d_u^{-1}$ be elements of the left cell $c_{0,u}$. Then
$\mu(d_wyw_0d_u^{-1}, d_{w'}zw_0d_u^{-1})=\mu(d_wyw_0,
d_{w'}zw_0)=\mu(d_wz^*w_0, d_{w'}y^*w_0)$, here $y^*=w_0y^{-1}w_0,
z^*=w_0z^{-1}w_0$. (We set $\mu(x,v)=\mu(v,x)$ if $v\le x$.)

We give some explanation for (c). By (a) we have
$$C_{d_uw_0y^{-1}d_w^{-1}}C_{
d_{w'}zw_0d_u^{-1}}=E_{d_u}C_{w_0y^{-1}d_w^{-1}}C_{
d_{w'}zw_0}F_{d_u}.$$ Note that
$$C_{w_0y^{-1}d_w^{-1}}C_{
d_{w'}zw_0}=\sum_{x\in X^+}h_{w_0y^{-1}d_w^{-1},
d_{w'}zw_0,xw_0}C_{xw_0}.$$ Using (a) we then get
$$C_{d_uw_0y^{-1}d_w^{-1}}C_{
d_{w'}zw_0d_u^{-1}}=\sum_{x\in X^+}h_{w_0y^{-1}d_w^{-1},
d_{w'}zw_0,xw_0}C_{d_uxw_0d_u^{-1}}.$$ So we have
$$h_{d_uw_0y^{-1}d_w^{-1},
d_{w'}zw_0d_u^{-1},d_uxw_0d_u^{-1}}=h_{w_0y^{-1}d_w^{-1},d_{w'}zw_0,xw_0}.$$

This implies

$$\gamma_{d_uw_0y^{-1}d_w^{-1},
d_{w'}zw_0d_u^{-1},d_uxw_0d_u^{-1}}=\gamma_{w_0y^{-1}d_w^{-1},d_{w'}zw_0,xw_0},$$
$$\delta_{d_uw_0y^{-1}d_w^{-1},
d_{w'}zw_0d_u^{-1},d_uxw_0d_u^{-1}}=\delta_{w_0y^{-1}d_w^{-1},d_{w'}zw_0,xw_0}.$$

If $w\ne w'$, then $d_wyw_0d_u^{-1}, d_{w'}zw_0d_u^{-1}$ are in
different right cells and $d_wyw_0, d_{w'}zw_0 $ are in different
right cells. So for any $x\in X^+$ we have
$$\gamma_{d_uw_0y^{-1}d_w^{-1},
d_{w'}zw_0d_u^{-1},d_uxw_0d_u^{-1}}=\gamma_{w_0y^{-1}d_w^{-1},d_{w'}zw_0,xw_0}=0.$$
By the formula of Springer in 1.3 and the above formula for $\delta
$, we get
$$\mu(d_wyw_0d_u^{-1}, d_{w'}zw_0d_u^{-1})=\delta_{d_uw_0y^{-1}d_w^{-1},
d_{w'}zw_0d_u^{-1},d_uw_0d_u^{-1}}$$
$$=\delta_{w_0y^{-1}d_w^{-1},d_{w'}zw_0,w_0}=\mu(d_wyw_0,
d_{w'}zw_0).$$

If $w=w'$  and $d_wyw_0d_u^{-1}\le d_{w'}zw_0d_u^{-1}$, then $y\le
z$. So $y^{-1}z$ is in the root lattice. Thus $l(z)-l(y)$ is even
since $l(y^{-1}z)$ is even. Therefore the values of $\mu$ in (c) are
all equal to 0 in this case. We have explained the first equality
for $\mu$ in (c)

 By (a) we get
$$C_{w_0y^{-1}d_w^{-1}}C_{ d_{w'}zw_0}=S_{y^*}S_z C_{w_0 d_w^{-1}}C_{
d_{w'} w_0},$$ $$C_{w_0(z^*)^{-1}d_w^{-1}}C_{
d_{w'}y^*w_0}=S_{(z^*)^*}S_{y^*} C_{w_0 d_w^{-1}}C_{ d_{w'} w_0}.$$

 Since $(z^*)^*=z$, using the formula of Springer in
1.3 we see that the second equality in (c) is true.

\section{Main results}

Now we can state our main results.

\begin{theorem} There is a positive integer $B$ such
that $\mu(y,w)\le B$ for all $y,w\in c_0$. In other words, the
Kazhdan-Lusztig coefficients $\mu(y,w)$ are bounded on $c_0\times
c_0$. (Recall $\mu(w,y)=\mu(y,w)$ if $y\le w$. Also, we set
$\mu(y,w)=\mu(w,y)=0$ if  $y\nleq w$ and $w\nleq y$.)
\end{theorem}

\begin{proof} Let $y,w\in c_0$ and assume that $\mu(y,w)\ne 0$. Then by
1.4(a) we have $y\el w$ or $y\er w$. It is no harm to assume that
$y\el w$. Thus we can find $v\in W_0$ and $y',w'\in W$ such that
$y=y'w_0d_v^{-1}$ and $w=w'w_0d_v^{-1}$ and
$l(y)=l(y')+l(w_0)+l(d_v^{-1})$ and $l(w)=l(w')+l(w_0)+l(d_v^{-1})$.
Using 2.2(c) we can see that $\mu(y,w)=\mu(y'w_0,w'w_0)$. Thus, to
prove the theorem we only need to show that $\mu$ is bounded on
$\gz\times\gz$. When $W$ has type $\tilde A_1$, it is easy to see
that $\mu(y,w)\le 1$ for all $y,w$ in $W$. In general let
$y=d_uxw_0,\ w=d_{u'}x'w_0\in \gz$, where $u,u'\in W_0$ and $x,x'\in
X^+$, be such that $\mu(y,w)\ne 0$.  If $y\er w$, then $u=u'$ and
$x\le x'$. This implies that $x'x^{-1}$ is  in the root lattice and
$l(y)-l(w)\equiv 0$(mod 2) since $y\el w$. This is impossible, so
$u\ne u'$ and $y$ and $w$ are not in the same right cell. Thus
$\gamma_{y^{-1},w,z}=0$ for any $z$ in $W$. By Springer formula in
1.3, we have $\mu(y,w)=\delta_{y^{-1},w,w_0}.$ Set
$x^*=w_0x^{-1}w_0\in X^+$. By 2.2 (a),
$C_{y^{-1}}C_w=S_{x^*}S_{x'}C_{w_0d_u^{-1}}C_{d_{u'}w_0}$. There are
only finitely many $z$ in $W$ such that
$h_{w_0d_u^{-1},d_{u'}w_0,z}\ne 0$ and if
$h_{w_0d_u^{-1},d_{u'}w_0,z}\ne 0$ then $z=z_1w_0$ for some $z_1\in
X^+$. Thus we have $\mu(y,w)=\sum_{z_1\in
X^+}m_{x^*,x',w_0z_1^{-1}w_0}\delta_{w_0d_u^{-1},d_{u'}w_0,z_1w_0}$.
Let $z_1^*=w_0z_1^{-1}w_0$. Then
$m_{x^*,x',w_0z_1^{-1}w_0}=\dim\Hom_G(V(x^*)\otimes
V(x'),V(z_1^*))=\dim\Hom_G(V(z_1)\otimes V(x'),V(x))\le\dim V(z_1)$.
Let $$B=\max\{\sum_{z_1\in X^+}\dim
V(z_1)\delta_{w_0d_u^{-1},d_{u'}w_0,z_1w_0}\,|\, u,u'\in W_0\}.$$
Then we have $\mu(y,w)\le B$. The theorem is proved.
\end{proof}

{\bf Remark:} When both $y$ and $w$ are in the left cell $\Gamma_0$,
the theorem is proved in \cite[\S7]{CPS2} by using representation
theory. Keeping the same assumption $y,w\in\Gamma_0$, a referee
pointed out that another proof
 of Theorem 3.1 can be obtained by observing that
the coefficients $\mu(y,w)$ are also the ``leading" coefficients of
the ``inverse" Kazhdan-Lusztig polynomials. Generically there are
only finitely many of these (according to \cite[Corollary 11.9]{L1})
and the non-generic ones are obtained by taking alternating sums
(see e.g. \cite[Theorem 2.2]{K}).

\subsection*{3.2} For the rest
of this section we assume that $G=SL_{n+1}(\bold C)$ and $W$ is the
corresponding extended affine Weyl group.

We number the simple roots $\alpha_1,\alpha_2,...,\alpha_n$ and the
simple reflections $s_0,s_1,...,s_n$ of $W$ as usual. Denote by
$x_1,...,x_n$ the fundamental weights. Let $\omega\in W$ be such
that $\omega s_i\omega^{-1}=s_{i+1}$ (we set $s_{n+1}=s_0$). Then
$x_1=\omega^{n} s_2\cdots s_ns_0$ and $x_n=\omega
s_{n-1}s_{n-2}\cdots s_1s_0$ and $x_1x_n=s_1s_2\cdots
s_{n-1}s_ns_{n-1}\cdots s_1s_0$. We also have
$$x_2=\omega^{n-1}s_4s_5\cdots s_ns_0s_1s_3s_4\cdots s_ns_0$$ and
$$x_{n-1}=\omega^{2}s_{n-3}s_{n-2}s_{n-4}s_{n-3}\cdots
s_0s_1s_ns_0.$$

\renewcommand{\thetheorem}{3.3}
\begin{theorem} Let $G=SL_{n+1}(\bold C)$ and let $W$ be the
corresponding extended affine Weyl group. Let
$$v=s_1s_ns_0s_2s_3\cdots s_{n-2}s_{n-1}s_{n-2}\cdots
s_2s_1s_ns_0$$ and let $x=\prod_{i=1}^nx_i^{a_i}\in X^+$ be a
dominant weight such that all $a_i\ge 2$. Then $\mu(xw_0,vxw_0)=n+2$
when $n\ge 4$.
\end{theorem}

\begin{proof} From 3.2 we see that $y=x_1x_nw_0\le vw_0$ and $l(vw_0)=l(y)+1$. So we have $\mu(y,vw_0)=1$. Let
$$w=s_{n-1}s_{n-2}\cdots s_2s_1 s_ns_{n-1}\cdots s_2s_4s_5\cdots s_n s_3s_4\cdots s_{n-1}s_1s_n.$$
Then  $R(w)=\{s_1,s_2,s_{n-1},s_{n}\}$ and $
wx_1x_2x_{n-1}x_n=s_0v=d_w$. Thus we have $v=s_0d_w\le d_w$. So
$vw_0v^{-1}$ is a distinguished involution (see 2.1). Noting that
$a(c_0)=l(w_0)$ and $l(v)=l(x_1x_n)+1$, we see that if
$\delta_{w_0,vw_0,xw_0}\ne 0$ for an $x\in X^+$, then
$l(x)+l(w_0)+a(w_0)-1\le l(w_0)+l(vw_0)$. So $l(x)\le l(v)+1=2n+2$.
Noting that $x$ is in the root lattice, we must have $x=x_1x_n$ or
$x=e$ (the neutral element). By 1.3, 2.2 (a) and 2.2 (b), we know
that $\delta_{w_0,vw_0,x_1x_nw_0}=1$. A direct computation shows
that $\delta_{w_0,vw_0,w_0}=2$ (see Lemma 3.4 below). Now let $x\in
X^+$, then we have $\mu(xw_0,vxw_0)=\sum_{z_1\in
X^+}m_{x^*,x,w_0z_1^{-1}w_0}\delta_{w_0,vw_0,z_1w_0}$
$=m_{x^*,x,x_1x_n}+2=m_{x_1x_n,x,x}+2$. When
$x=\prod_{i=1}^nx_i^{a_i}$ with all $a_i\ge 2$, we have
$m_{x_1x_n,x,x'}=0$ if $x'x^{-1}$ is not a weight of $V(x_1x_n)$,
and $m_{x_1x_n,x,x\lambda}=\dim V(x_1x_n)_\lambda$, the dimension of
the $\lambda$-weight space of $V(x_1x_n)$. In particular, when
$\lambda=e$, the neutral element, we get $m_{x_1x_n,x,x}=\dim
V(x_1x_n)_e=n$, the dimension of the maximal torus of the Lie
algebra $sl_{n+1}(\bold C)$. The theorem is proved.
\end{proof}

\renewcommand{\thetheorem}{3.4}
\begin{lemma} Let $v$ be as in Theorem 3.3. If $n\ge 4$ then
$$\mu(w_0,vw_0)=\delta_{w_0,vw_0,w_0}=2.$$
\end{lemma}

\begin{proof} Note that $c_0$ is the lowest two-sided cell and
$\Gamma_0$ is a left cell in $c_0$. This implies that for
$y,w\in\Gamma_0$ with $y\le w$, we have
$$P_{y,w}=q^aP_{sy,sw}+q^{1-a}P_{y,sw}-\sum_{\st {\st {z\in\Gamma_0}{y\le
z<sw}}{sz<z}}\mu(z,sw)q^{\frac12(l(w)-l(z))}P_{y,z},$$ where $s\in
S$ such that $sw\le w$, $a=0$ if $sy\le y$ and $a=1$ if $sy\ge y$.

We apply this formula to compute $\mu(w_0,vw_0)$.

Let $v_1=s_1v$. Note that $P_{sy,w}=P_{y,w}$ if $sw\le w$ for $s\in
S$. We have
$$P_{w_0,vw_0}=(1+q)P_{w_0,v_1w_0}-\sum_{\st {\st {z\in\Gamma_0}{w_0\le
z<v_1w_0}}{s_1z<z}}\mu(z,v_1w_0)q^{\frac12(l(vw_0)-l(z))}P_{w_0,z}.$$

We claim that

\noindent(1) If $z\in\Gamma_0,\ {w_0\le z<v_1w_0}$ and ${s_1z<z}$,
then the degree of the polynomial
$\mu(z,v_1w_0)q^{\frac12(l(vw_0)-l(z))}P_{w_0,z}$ is less than
$\frac12(l(vw_0)-l(w_0)-1)=n$.

Note that $v_1=s_ns_0s_{n-1}s_{n-2}\cdots s_2s_1 s_3s_4\cdots
s_ns_0$. Thus  if $z\in\Gamma_0,$\ ${w_0\le z<v_1w_0}$, then $z$ is
one of the following elements:

\noindent$w_{ij}=s_ns_0s_is_{i-1}\cdots s_2s_1s_js_{j+1}\cdots
s_ns_0w_0$,\quad $1\le i\le n-1,\ 3\le j\le n$;

\noindent$u_{ij}= s_0s_is_{i-1}\cdots s_2s_1s_js_{j+1}\cdots
s_ns_0w_0$,\quad $1\le i\le n-1,\ 3\le j\le n$;

\noindent$v_{ij}=s_is_{i-1}\cdots s_2s_1s_js_{j+1}\cdots
s_ns_0w_0$, \qquad$1\le i\le n-1,\ 3\le j\le n$;

\noindent$w_{i}=s_is_{i-1}\cdots s_2s_1s_0w_0$, \qquad$1\le i\le
n-1;$

\noindent$u_{j}=s_js_{j+1}\cdots s_ns_0w_0$, \qquad$2\le j\le n$;

\noindent $s_0w_0, \ w_0$.

Note that $s_iv_1w_0\le v_1w_0$ for $i=2,3,...,n-2$. If $i\ge 2$,
then $s_1w_{ij}\ge w_{ij}$. If $n=1,$ then $s_2w_{ij}\le w_{ij}$,
thus $\mu(w_{ij},v_1w_0)=0$.  (We also have: when $n-2\ge i$ and
$j=n$, then $w_{ij}$ is not in $\Gamma_0$;  if $i\le n-3$, then
$s_{i+1}w_{ij}\ge w_{ij}$, so $\mu(w_{ij},v_1w_0)=0$.)

Now we consider $v_{ij}$. When $i\le n-3$, we have
$s_{i+1}v_{ij}\ge v_{ij}$, so $\mu(v_{ij},v_1w_0)=0$.  If $i=n-2$
and $4\le j\le n-1$, then $s_{j-2}v_{ij}\ge v_{ij}$, so
$\mu(v_{ij},v_1w_0)=0$. If $i=n-2$ and $j=3$, then $s_1v_{ij}\ge
v_{ij}$. If $i=n-2$ and $j=n$, then the degree of $P_{w_0,v_{ij}}$
is 1, less than $\frac12(l(v_{ij}-l(w_0)-1)=\frac{n-1}2$ since
$n\ge 4$. If $i=n-1$ and $4\le j\le n$, then $s_{j-2}v_{ij}\ge
v_{ij}$, so $\mu(v_{ij},v_1w_0)=0$. If $i=n-1$ and $j=3$, then
$s_1v_{ij}\ge v_{ij}$.

We have $\mu(w_0,u_{ij})=0$ since $s_0u_{ij}\le u_{ij}$ and
$s_0w_0\ge w_0$ and $l(u_{ij})-l(w_0)>1$. We have
$\mu(w_i,v_1w_0)=0$ since $s_nv_1w_0\le v_1w_0$ and $s_nw_i\ge w_i$
and $l(v_1w_0)-l(w_i)>1$. Note that $s_1u_j\ge u_j$ and
$s_1s_0w_0\ge s_0w_0$. Also we have $\mu(w_0,v_1w_0)=0$ since
$l(v_1)$ is even.

Thus we have shown that  statement (1) is true.

Let $v_2=s_nv_1$.  Note that $P_{w_0,v_2w_0}=P_{s_0w_0,v_2w_0}.$
We have
$$P_{w_0,v_1w_0}=(1+q)P_{s_0w_0,v_2w_0}-\sum_{\st {\st {z\in\Gamma_0}{w_0\le
z<v_2w_0}}{s_nz<z}}\mu(z,v_2w_0)q^{\frac12(l(v_1w_0)-l(z))}P_{w_0,z}.$$

Using a similar argument for (1) we see that

\noindent(2) If $z\in\Gamma_0,\ {w_0\le z<v_2w_0}$ and ${s_nz<z}$,
then the degree of the polynomial
$\mu(z,v_2w_0)q^{\frac12(l(v_1w_0)-l(z))}P_{w_0,z}$ is less than
$\frac12(l(v_1w_0)-l(w_0)-2)=n-1$.

Let $z_i=s_0s_is_{i-1}\cdots s_1 s_3s_4\cdots s_ns_0w_0$. By a
direct computation, we get

\noindent(3) $P_{s_0w_0,z_1}=P_{s_1s_0w_0,z_1}=1+q$ and
$P_{s_0w_0,z_2}=1+2q+q^2.$

Let $u_2=s_2s_3\cdots s_ns_0w_0$, by a direct computation we get

\noindent(4) $P_{u_2,z_i}=1+q$ if $i\ge 3$ and $P_{u_2,z_2}=1$.

Using this it is not difficult to get the following formula.

\noindent(5) $P_{s_0w_0,z_i}=(1+q)P_{s_0w_0,z_{i-1}}-
qP_{s_0w_0,z_{i-2}}-\mu(s_0w_0,z_{i-1})q^{\frac12(n+i-1)}-\mu(u_2,z_{i-1})q^{\frac12i}.$

Thus we have
$$P_{s_0w_0,z_i}=\begin{cases} q^i+2q^{i-1}+\text{lower degree terms\quad
if }n>i+1\\ 2q^{i-1}+\text{lower degree terms\qquad if
}n=i+1.\end{cases}$$

Note that $v_2w_0=z_{n-1}$. Thus, we have
$P_{s_0w_0,v_2}=2q^{n-2}$+lower degree terms. Now using (1) and (2)
we see that the lemma is true.
\end{proof}

\medskip
\section{Some consequences}

In this section we shall assume that $G$ is simply connected and
simple. We shall write the operation of $X$ additively. For $x\in
X$, denote by $t_x$ the translation $y\to y+x$ of $X$. Let
$\alpha_0$ be the highest short root of $R$. Set
$s_0=s_{\alpha_0}t_{p\alpha_0}$ (recall that we use $s_\alpha$ for
the reflection on $E=X\otimes \bold R$ corresponding to $\alpha\in
R$). Let $W'$ be the subgroup of $GL(E)$ generated by all $s_\alpha$
($\alpha\in R$) and $s_0$. Then $W'$ is an affine Weyl group and is
isomorphic to  the group $W_0\ltimes pQ$, where $Q$ stands for the
root lattice. Here $p$ could be any positive (or negative) integer,
but it is convenient here to suppose $p$ is a fixed prime.

We shall further identify $W'$ with the affine Weyl group $W$
defined in \cite[section 1.1]{L1} and let $W'$ act on $X$ through
the affine Weyl group defined in loc.cit. We denote this action by
$*$. Then in terms of this new action, for $w\in W'$, we have
$w*(-\rho)=w^{-1}(-\rho)$, and $w$ is in $\Gamma_0$ if and only if
$w*(-\rho)-\rho$ is dominant, where $\rho$ is the sum of all
fundamental weights. Note that $w*(-\rho)=w^{-1}(-\rho)$ is just a
fact; it does not imply that $w*(u*(-\rho))=w^{-1}(u*(-\rho))$ for
$w,u$ in $W'$.

Now let $G=SL_{n+1}(\bold C)$ and we number the simple reflections
of $W_0$ as usual. Let $v$ be as in Theorem 3.3 and
$\beta=\alpha_2+\cdots+\alpha_{n-1}$. Then we have $v=s_{\beta}
t_{p\varpi_2+p\varpi_{n-1}}$ (we use $\varpi_i$ for the fundamental
weight corresponding to the simple root $\alpha_i$). Let
$\lambda=t_{2p\rho}w_0*(-\rho)-\rho=2p\rho$ and
$\mu=vt_{2p\rho}w_0*(-\rho)-\rho$. We have
$\mu=w_0t_{-2p\rho}t_{-p\varpi_2-p\varpi_{n-1}}s_\beta
(-\rho)-\rho$=$2p\rho+(n-2)(\varpi_1+\varpi_{n})+(p-n+2)(\varpi_2+\varpi_{n-1})$.
We have $\langle\lambda+\rho,\alpha_0^\vee\rangle=2pn+n$ and
$\langle\mu+\rho,\alpha_0^\vee\rangle=2pn+2p+n.$ The Jantzen region
is defined to be the set of all vectors $\nu$ with $0\le
\langle\nu+\rho,\alpha_0^\vee\rangle\le p(p-h+2)$, where $h$ is the
Coxeter number. For $G=SL_{n+1}(\bold C)$, we know $h=n+1$. Thus if
$p\ge 3n+2$, then $p(p-h+2)=p(p-n+1)\ge 2pn+3p> 2pn+2p+n.$ Thus both
$\lambda$ and $\mu$ are in the Jantzen region.

Now replace $\bold C$ with an algebraically closed field $k$ of
characteristic $p$ and let $H$ be  a simply connected and simple
algebraic group  over  $k$. It is known for each root system that
when $p$ is sufficiently large, Lusztig's conjecture for modular
representations of algebraic groups (Lusztig's modular conjecture in
short, see \cite{L0} for the formulation) is true for irreducible
modules of $H$ with highest weight in the Jantzen region.

For $H=SL_{n+1}(k)$, by Theorem 3.3 we know that the Kazhdan-Lusztig
coefficient $\mu(t_{2p\rho}w_0,vt_{2p\rho}w_0)$ of
$P_{t_{2p\rho}w_0,vt_{2p\rho}w_0}$ is $n+2$ if $n\ge 4$. It is known
that the coefficient is related to the dimension of the first
extension groups for extensions between certain irreducible modules
for $H$. See for example \cite{CPS2} and references
therein.\footnote{It is useful to note, in comparing notation, that
$P_{y,w}=P_{y^{-1},w^{-1}}$, and so
$\mu({y,w})=\mu({y^{-1},w^{-1}})$, for all $y<w$ in W. See
\cite{KL}.} So we conclude that the dimension goes to infinity when
$n$ increases.

\noindent{\bf Remark:} We do not know if the quantities $\mu(w_0,w)$
that correspond to 1-cohomology dimensions can go to infinity. We do
have some plausible candidates, though. Let $\omega\in W_0\ltimes
pX$ be such that $\omega(s_i)=s_{i+1}$ for any $0\le i\le n$ (set
$s_{n+1}=s_0$) for any $0\le i\le n$. Assume that $n=2k-1$ is odd.
Then $w=s_k\omega^kt_{-p\rho}$ is in $W'\cap \Gamma_0$, and the
weight $w*(-\rho)-\rho$ is $p$-restricted. The explicit form of
$w*(-\rho)-\rho$ is $(p-2)\rho -\alpha_0 - (p-n-1)\varpi_k$. For
$n=5$ and $p$ large, this is the same weight which gave 1-cohomology
dimensions  of 3 with irreducible coefficients in \cite{S} (found
there by computer calculations
 of Kazhdan-Lusztig polynomials).  For general odd $n>1$,
the weight is in the second top alcove of the fundamental $p$-box
$\{ v\in X\otimes\mathbf{R}\ |\ 0<\langle
v+\rho,\alpha_i^\vee\rangle<p,\ 1\le i\le n\}$. If $n>1$ is odd,
then $l(t_{-p\rho})-l(w_0)$ is even, so $\mu(w_0,t_{-p\rho})=0$.
Thus $w*(-\rho)-\rho$ is the largest possible $p$-restricted weight
in the orbit $W*(-\rho)-\rho$ whose corresponding irreducible module
could have nonzero 1-cohomology. For $n=4k\ge 4$, set
$w=s_0t_{-p\rho}$. Then $w*(-\rho)-\rho=(p-2)\rho-(p-n)\alpha_0$ is
also the largest possible $p$-restricted weight whose corresponding
irreducible module could have nonzero 1-cohomology. In fact, when
$n=4$, $w=s_0t_{-p\rho}$ is just the $vw_0$ in Lemma 3.4. The
question is whether $\mu(w_0,w)$ goes to infinity when $k$
increases? (When $n=4k+2$, we have not found a similar candidate.)

\medskip
\def\l{\lambda}

\section{Representation-theoretic argument}

\setcounter{equation}{0} \setcounter{theorem}{0}
\setcounter{remark}{0}

The aim of this section is to prove a version of Theorem 3.3 by a
representation-theoretic argument. The result is weaker, in that we
do not exactly compute the relevant Kazhdan-Lusztig coefficents,
 but only get a good lower bound. However, the hypotheses required are also somewhat weaker.
 As in $\S$4, we will assume $p$ is sufficiently large so that the Lusztig modular conjecture holds for
 the group $G = SL_{n+1}(k)$, with $k$ an algebraically closed field of characteristic $p$.
 We also require $p \geq 3n+2$, as in $\S$4, at least to begin our discussion.
  This assumption was used in $\S$4 to ensure the weights $\lambda$ and $\mu$ there were in
  the Jantzen region. Eventually, in this section, we will be able to obtain a good lower bound for more weights,
  and we will only need that $p \geq 3n$ to enable some of them to be in the Jantzen region.
  This latter inequality, fortunately, also ensures that a formula of Andersen may be applied.

\subsection{Andersen's Formula}For dominant weights $\lambda = \lambda_0 + p \lambda_1$ and $\mu = \mu_0 + p \mu_1$ with $\lambda_0 \neq \mu_0$ both
restricted and $\lambda_1$, $\mu_1$ both dominant, the formula asserts (for any connected semisimple group) that

\renewcommand{\thetheorem}{5.1}
\begin{theorem}
Assume $p \geq 3h-3$ (which is $3n$ for type $A_n$).  Then
$$\mbox{{\textrm{dim Ext}}}^1_G(L(\lambda),L(\mu))$$
$$ =  \sum_\nu \mbox{\textrm{dim Ext}}^1_G(L(\lambda_0 + p \nu),L(\mu_0)) \mbox{\textrm{dim Hom}}_G(L(\lambda_1), L(\nu) \otimes L(\mu_1)).
$$
\end{theorem}

Here $\nu$ ranges over all dominant weights satisfying $\lambda_0 +
p \nu \leq 2(p-1)\rho+w_0(\mu_0)$. For a more complete statement,
see \cite[7.8]{CPS2} or \linebreak \cite{A2}.

\subsection{A lower bound for some $\text{Ext}^1$ dimensions}
Recall that $\varpi_1,\ldots,\varpi_n$ are the   fundamental weights
(indexed corresponding to $\alpha_1, \ldots,\alpha_n$). Set $a =
n-2$ and put $\mu_0 = a( \varpi_1+ \varpi_n ) + (p-a)(\varpi_2+
\varpi_{n-1})$ (note that this is the dominant weight
$vw_0*(-\rho)-\rho$ corresponding to $vw_0$, see Section 4). Let
$\mu_1$ be any weight of the form $\sum^n_{i=1} a_i\varpi_i$ with
each $a_i \geq 2$, and put $\mu = \mu_0+p \mu_1$. (Thus, if $\mu_1$
is $2\rho$, this is the same $\mu$ as in $\S$4.) We will apply
Andersen's formula in  5.1 to $\mu$ and to $\lambda = p \mu_1$. The
following result is a weak version of Theorem 3.3, but still strong
enough to show $\mbox{dim Ext}^1_G(L(\lambda),L(\mu)) \rightarrow
\infty$ as $n \rightarrow \infty$.

\renewcommand{\thetheorem}{5.2}
\begin{theorem}
Let $\lambda, \mu$ be as above, and assume $n + \sum^n_{i=1} a_i
\leq p$ (which certainly holds if $\mu_1 = 2\rho$). Then
$$
\mbox{dim Ext}^1_G (L(\lambda), L(\mu)) \geq n.
$$
\end{theorem}

\begin{proof}
First, we show $\mbox{\mbox{Ext}}^1_G(L(\lambda_0 + p
\alpha_0),L(\mu_0)) \neq 0$.  Here $\lambda_0$ is the zero weight,
and $\alpha_0 = \varpi_1+ \varpi_n$. In particular,
$$
L(\lambda_0 + p \alpha_0) \cong L(p \alpha_0),
$$
and
$$
\mbox{\mbox{Ext}}^1_G(L(\lambda_0+p \alpha_0),L(\mu_0)) \cong
\mbox{\mbox{Ext}}^1_G(L(p \alpha_0),L(\mu_0)).
$$
Observe that the weight $\mu_0$ is obtained from $p \alpha_0$ by
reflection in the hyperplane
$$
\left\{ x \in \mathbb{R}^n| \ \langle x + \rho,(\alpha_2 +\ldots+
\alpha_{n-1})^\vee \rangle = p\right\}.
$$
We can also directly calculate  the number $d(\mu_0)$ of hyperplanes
of the form $\left\{x \in \mathbb{R}^n| \ \langle x +
\rho,\alpha^\vee \rangle = mp\right\}$, with $\alpha > 0$ and $m \in
\mathbb{Z}$, which separate $\mu_0$ from 0, and compare $d(\mu_0)$
with $d(p \alpha_0)$. We find that $\mu_0$ and $p \alpha_0$ are on
opposite sides of only the hyperplanes defined by $\alpha =
\alpha_1$, $m =1$;
 $\alpha = \alpha_n$, $m = 1$;
 $\alpha = \alpha_0 - \alpha_1$, $m = 2$; $\alpha = \alpha_0-\alpha_n$,
 $m =2$; $\alpha = \alpha_0-\alpha_1 - \alpha_n$, $m =1$.
 For the first two of these five hyperplanes, the weight $\mu_0$ is on the same side as 0.
 For the last three, $p \alpha_0$ is on the same side as 0.  Thus, $d(\mu_0) = d(p \alpha_0)+1$.
 It now follows from \cite[II, 6.24]{J} that $L(p\alpha_0)$ is a composition factor of the costandard
 module $\nabla(\mu_0) = H^0(\mu_0)$.
  Also, the strong linkage principle implies that $p \alpha_0$ is maximal among the highest weights
  of composition factors of $\nabla(\mu_0)/L(\mu_0)$.
  Thus, there is a nonzero homomorphism $\Delta(p \alpha_0) \rightarrow \nabla (\mu_0)/L(\mu_0)$, and
  this shows $\mbox{Ext}^1_G(\Delta (p \alpha_0), L(\mu_0)) \neq 0$.  However, since $p$ is large enough that the Lusztig conjecture
  holds, and both $p \alpha_0$ and $\mu_0$ lie in the Jantzen region (we calculate
this using $p>n$), we have
$$
\mbox{Ext}^1_G(L(p \alpha_0),L(\mu_0)) \cong \mbox{Ext}^1_G(\Delta(p
\alpha_0),L(\mu_0))
$$
by a well-known result of Andersen \cite[Proposition 2.8]{A1}. (We
take the $\lambda, y, w ,s$ there to be $-2\rho, t_{p\alpha_0}w_0,
s_2vw_0,s_2$ respectively. Note that $s_2vw_0, s_2$ can be replaced
by $s_{n-1}vw_0,s_{n-1}$ respectively.) This proves the claim.

To prove the theorem, it is now sufficient by Andersen's formula in
  5.1, to show
$$
\mbox{dim Hom}_G(L(\alpha_0),L(\mu_1) \otimes L(\mu_1)) \geq n.
$$
Note that $ \alpha_0 $ and $ \mu_1 $ lie in the closure of the
lowest $p$-alcove. (This uses our inequality $\sum^n_{i=1} a_i + n
\leq p$.)  Thus
$$
L(\alpha_0) \cong \Delta (\alpha_0) \cong \nabla(\alpha_0), \quad
L(\mu_1) \cong \Delta(\mu_1) \cong \nabla (\mu_1),
$$
and
\begin{eqnarray*}
\mbox{Hom}_G(L(\alpha_0),L(\mu_1) \otimes L(\mu_1)) & \cong &
\mbox{Hom}_G(\Delta(\mu_1),\Delta(\alpha_0) \otimes \nabla(\mu_1))
\\[3mm]
& \cong & \mbox{Hom}_G(\Delta(\mu_1),\mbox{Ind}_B^G(\Delta(\alpha_0)
\otimes k(\mu_1)),
\end{eqnarray*}
where $B$ is the Borel subgroup associated to the negative roots,
and $k(\mu_1)$ is the one-dimensional $B$ module with weight
$\mu_1$. The module $\Delta(\alpha_0)$ is just the Lie algebra
$\mathfrak{g}$ of $G$, with the usual adjoint action. The Borel
subalgebra $\mathfrak{b}$ with negative root spaces is a
$B$-submodule, and $\mathfrak{b} \otimes k(\mu_1)$ is a
$B$-submodule of $\mathfrak{g} \otimes k(\mu_1) = \Delta(\alpha_0)
\otimes k(\mu_1)$. This $B$-submodule has a $B$-quotient which is a
direct sum of $n$ copies of $k(\mu_1)$ and all the other composition
factors have the form $k(\mu_1-\beta)$, where $\beta$ is a positive
root. Because of our assumption that all $a_i$ are at least 2, the
weights $\mu_1 - \beta$ are all dominant. (It is possible to carry
through a version of  the argument which follows with a requirement
weaker than "dominant", so it is actually enough that each $a_i$ be
at least 1. See 5.3(c) below.) Applying Kempf's vanishing theorem,
we find that $\mbox{Ind}^G_B(\mathfrak{b} \otimes k(\mu_1))$ has a
quotient isomorphic to a direct sum of $n$ copies of
$\nabla(\mu_1)$. The kernel of the map to this quotient is filtered
by modules $\nabla(\mu_1-\beta)$.  Since $\nabla(\mu_1) \cong
\Delta(\mu_1)$, and there are no nontrivial extensions of a standard
module by a costandard module, the quotient map is split. Since
$\mbox{Ind}^G_B(\mathfrak{b} \times k(\mu_1)) \subseteq
\mbox{Ind}^G_B(\Delta(\alpha_0) \otimes k(\mu_1))$, we have
\begin{eqnarray*}
\lefteqn{\hspace*{-20pt} \mbox{dim
Hom}_G(\Delta(\mu_1),\mbox{Ind}^G_B(\Delta(\alpha_0) \otimes
k(\mu_1))
} \\[2mm]
& \geq &\mbox{dim Hom}_G(\Delta(\mu_1),\mbox{Ind}^G_B(\mathfrak{b}
\otimes k(\mu_1))) \geq n.
\end{eqnarray*}
Taken with the isomorphisms and discussions above, this completes
the proof of the theorem.
\end{proof}

\subsection{Remarks}
(a) The assumption $n + \sum^n_{i=1} a_i \leq p$ is used only to
guarantee $L(\mu_1) \cong \Delta(\mu_1) \cong \nabla(\mu_1)$. In
\cite[\S 7]{CPS2} it is suggested that Andersen's formula should be
true, appropriately formulated (and with a similar proof) for
quantum enveloping algebras at a root of unity. In such a
formulation, the terms involving $\mbox{Hom}_G$ would instead
involve a Hom over the ordinary characteristic 0 enveloping algebra
of $\mathfrak{g}$. Thus, the required isomorphisms on $L(\mu_1)$
would hold without the assumed inequality. That is, Theorem 5.2
should hold at a $p^{th}$ root of unity without the assumed
inequality (if $p > n$). Essentially, use of the quantum group frees
the representation theory from dependence on the Jantzen region and
allows a 1--1 correspondence between affine Weyl group results and
representation theory results in our context.

(b) If Lemma 3.4 is interpreted as an $\mbox{\mbox{Ext}}^1$ result
and fed into Andersen's formula, it enables $n$ in the inequality in
Theorem 5.2 to be replaced with $n+2$.  This result is almost as
good as Theorem 3.3, though the latter gives the resulting
inequality as an equality.

(c) The hypotheses $a_i \geq 2$ ($i = 1,\ldots,n$) in 5.2 (and in
the auxiliary remarks 5.3(a), 5.3(b) above) {\em can be weakened to
just assuming} $a_i \geq 1$ ($i = 1,\ldots,n$).  To see this, note
that $a_i \geq 2$ condition was used only to guarantee that the
weights $\mu_1-\beta$ were all dominant, with $\beta$ any positive
root.
 The dominance guaranteed, through Kempf's theorem, that the higher derived functors $R^1\mbox{Ind}^G_B(k(\mu_1-\beta))$ were zero.
However, this is true also when $\langle
\mu_1-\beta,\alpha_i^\vee\rangle = -1$ for some $i$. (All
$R^j\mbox{Ind}^G_B(k(\mu_1-\beta))$ vanish in this case; see
\cite[II,5.4(a)]{J}.)
 We find, with the assumption $a_i \geq 1$, for $1 \leq i \leq n$, that the kernel of the map from $\mbox{Ind}^G_B(\mathfrak{b} \otimes k(\mu_1))$
onto a direct sum of $n$ copies of $k(\mu_1)$, is again filtered by
costandard modules. (Any potential section
$\mbox{Ind}^G_Bk(\mu_1-\beta)$, in which $\mu_1-\beta$ is not
dominant, is just zero.)

Thus, assuming that Andersen's formula extends to the quantum case,
as discussed in 5.3(a), we have, in the notation of Theorem 3.3,
$$
\mu(x w_0,vxw_0) \geq n+2,
$$
assuming, as in Theorem 3.3, that $n \geq 4$, but weakening the
requirement $a_i \geq 2$ to $a_i \geq 1$, for all $i =
1,2,\ldots,n$. The assumption $p \geq 3n+2$ can be replaced, then,
with $p \geq 3n$.  (Still $p$ must be large enough for the Lusztig
conjecture to hold.)

(d) Andersen observed that the inequality in Theorem 5.2 can be
improved to an equality. More precisely, for all dominant weights
$\mu_1$, if $p>n+1$ and $p$ is large enough that the Lusztig modular
conjecture holds for $G$, one has
$$\text{dim Ext}^1_G(\Delta(\mu_1)^{(1)}, L(\mu_0) \otimes \nabla(\mu_1)^{(1)}) =
  n+2 - f(\mu_1),$$
  where $f(\mu_1) $ is the number of simple roots orthogonal to
  $\mu_1.$ To see this one first notes that
 $H^1(G_1, H^0(\mu_0))$ = 0
  (\cite{AJ}), where $G_1$ is the first Frobenius kernel of $G$. Here the cohomology group $H^0(\mu_0)$ is the $G$-module $\nabla(\mu_0)$. Then
  $$H^1(G_1, L(\mu_0))= H^0(G_1, H^0(\mu_0)/L(\mu_0)) = L(\alpha_0)
  \oplus  k\oplus k.$$
  Here the first summand comes as in 5.2, and
the appearance of two copies of $k$ is a consequence of Lemma 3.4
  (which  gives $H^1(G, L(\mu_0)) = k\oplus k).$
Since $M = L(\alpha_0) \otimes \nabla(\mu_1)$ has a
  good filtration $M=M_0\supset M_1\supset M_2\supset\cdots\supset M_{r-1}\supset M_r=0$ (this uses
$L(\alpha_0)= \nabla(\alpha_0)$, which holds for $p$ prime to
$n+1$), the dimension of $\text{Hom}_G(\Delta(\mu_1), M)$ is equal
to the
  number of occurrences  of $\nabla(\mu_1)$ in subquotients $M_0/M_1,$ $M_1/M_2,$ ..., $M_{r-1}/M_r$ of any good filtration of $M$, i.e. $n -f(\mu_1).$


\noindent{\bf Acknowledgement:} Part of the work was done during
Xi's visit to the Department of Mathematics, Osaka City University
and to the Department of Mathematics, University of Virginia. Xi is
very grateful to the departments for hospitality and for partial
financial support. The authors are grateful to H. Andersen and J.
Humphreys for helpful comments. Finally we would like to thank the
referees for detailed and valuable comments.

\bibliographystyle{unsrt}

\begin{thebibliography}{AAAA}

\bibitem[A1]{A1} H. H. Andersen, {\sl An inversion
 formula for the Kazhdan-Lusztig polynomiaals for affine Weyl groups.} Adv. in Math 60 (1986), 125--153.

\bibitem[A2]{A2} H. H. Andersen, {\sl
Extensions of modules for algebraic groups.} Amer. J. Math. 106
(1984), 489--504.

\bibitem[AJ]{AJ} H. H. Andersen, J. C.  Jantzen, {\sl
Cohomology of induced representations for algebraic groups.} Math.
Ann. 269 (1984), no. 4, 487-525.


\bibitem[AJS]{AJS} H. H. Andersen, J. C.  Jantzen, W. Soergel, {\sl Representations of quantum
groups at a $p$th root of unity and of semisimple groups in
characteristic $p$: independence of $p$.}  Ast\'erisque  No. 220
(1994), 321 pp.

\bibitem[CPS1]{CPS1} E. Cline, B. Parshall, L. Scott, {\sl Detecting rational cohomology of
algebraic groups.} J. London Math. Soc. (2) 28 (2) (1983) 293--300.



\bibitem[CPS2]{CPS2} E. Cline, B. Parshall and L. Scott, {\sl Reduced
standard modules and cohomology.} Trans. Amer. Math. Soc.  361
(2009),  no. 10, 5223--5261.


\bibitem[G]{G} R. M. Guralnick, {\sl The dimension of the first cohomology
group.}
In: Representation Theory, II, Ottawa, ON, 1984Lecture Notes in
Math. 1178, Springer-Verlag, Berlin, pp. 94¨C97.

\bibitem[IM]{IM} N. Iwahori and H. Matsumoto, {\sl On some Bruhat
decomposition and the structure of Hecke rings of p-adic Chevalley
groups.} Publ. IHES 25 (1965), 5-48.

\bibitem[J]{J} J.C. Jantzen,
{\sl Representations of Algebraic Groups.} 2nd ed., Mathematical
Surveys and Monographs, 107, AMS, 2003, xiv+576pp.

\bibitem[K]{K} Masaharu Kaneda, {\sl On the inverse Kazhdan-Lusztig polynomials for
affine Weyl groups.} J. Reine Angew. Math. 381 (1987), 116--135.


\bibitem[KL]{KL} D. Kazhdan and G. Lusztig, {\sl Representations of Coxeter
groups and Hecke algebras.} Invent. Math. 53 (1979), 165-184.

\bibitem[L0]{L0} G. Lusztig, {\sl Some problems in the representation theory of
finite Chevalley groups.} Proc. Symp. Pure Math. 37 (1980),
313--317.

\bibitem[L1]{L1} G. Lusztig, {\sl Hecke algebras and Jantzen's generic decomposition patters.}
 Adv. in Math. 37 (1980), 121-164.

\bibitem[L2]{L2} G. Lusztig, {\sl Singularities, character formulas, and a
$q$-analog of weight multiplicities.} Ast\'erisque 101-102 (1983),
pp.208-227.

\bibitem[L3]{L3} G. Lusztig, {\sl Cells in affine Weyl groups.} in
``Algebraic groups and related topics", Advanced Studies in Pure
Math., vol. 6, Kinokunia and North Holland, 1985, pp. 255-287.

\bibitem[L4]{L4} G. Lusztig, {\sl Cells in affine Weyl groups, II.} J. Alg.
109 (1987), 536-548.

\bibitem[L5]{L5} G. Lusztig, {\sl Cells in affine Weyl groups, III.}
J. Fac. Sci. Univ. Tokyo Sect. IA Math. 34 (1987), 223-243.


\bibitem[L6]{L6} G. Lusztig, {\sl Nonlocal finiteness of a $W$-graph.}
Represent. Theory 1 (1996), 25-30.

\bibitem[MW]{MW} T. McLarnan and G. Warrington, {\sl Counterexamples to the
0,1-Conjecture.} Represent. Theory 7 (2003), 181-195.

\bibitem[S]{S} L. Scott, {\sl Some new examples in 1-cohomology.}
Special issue celebrating the 80th birthday of Robert Steinberg.
J. Alg. 260 (2003), no. 1, 416--425.

\bibitem[Sh1]{Sh1} J.-Y. Shi, {\sl A two-sided cell in an affine Weyl group I.}
J. London Math. Soc. 37 (1987), 407-420.

\bibitem[Sh2]{Sh2} J.-Y. Shi, {\sl A two-sided cell in an affine Weyl group II.}
J. London Math. Soc. 38 (1988), 235-264.

\bibitem[Sp]{Sp} T. A. Springer, {\sl Letter to G. Lusztig.}
1987.

\bibitem[W]{W}
Liping Wang,  {\sl Leading coefficients of the Kazhdan-Lusztig
polynomials for an affine Weyl group of type $\tilde{B}_2$.}
preprint; arXiv: 0805.3463v1.

\bibitem[X1]{X1} N. Xi, {\sl The based ring of the lowest two-sided cell of
an affine Weyl group.} J. Alg. 134 (1990), 356-368.

\bibitem[X2]{X2} N. Xi, {\sl Representations of Affine Hecke
Algebras.} Lecture Notes in Mathematics 1587, Springer-Verlag,
Berlin Heidelberg 1994.

\bibitem[X3]{X3} N. Xi, {\sl The leading coefficient of certain Kazhdan-Lusztig
polynomials of the permutation group $S_n$.} J. Algebra 285 (2005),
no. 1, 136--145.





\end{thebibliography}

\end{document}